\documentclass{amsart}
\usepackage{amsfonts}

\newtheorem{theorem}{Theorem}
\theoremstyle{plain}

\newtheorem{corollary}{Corollary}

\newtheorem{definition}{Definition}
\newtheorem{example}{Example}

\newtheorem{lemma}{Lemma}

\newtheorem{proposition}{Proposition}

\numberwithin{equation}{section}
\input{tcilatex}

\begin{document}
\title{Generalization of $([e],[e]\vee \lbrack c])$\textbf{-Ideals Of
BE-algebras}}
\author{Ahmad Fawad Ali}
\curraddr{Department of Basic Sciences, Riphah Internaional University
Islamabad Pakistan.}
\email{afawada@gmail.com}
\author{Saleem Abdullah}
\curraddr{Department of mathematics, Quaid-e-Aazam University Islamabad
Pakistan. }
\email{saleemabdullah81@yahoo.com}
\author{Muhammad Sarwar Kamran}
\address{Department of Basic Sciences, Riphah Internaional University
Islamabad Pakistan}
\email{drsarwarkamran@gmail.com}
\author{Muhammad Aslam}
\address{Department of Mathematics, King Khlid University Saudi Arabia.}

\begin{abstract}
In this paper, using $N$-structure, the notion of an $N$-ideal in a
BE-algebra is introduced. Conditions for an $N$-structure to be an $N$-ideal
are provided. To obtain a more general form of an $N$-ideal, a point $N$%
-structure which is ($k$ conditionally) employed in an $N$-structure is
proposed. Using these notions, the concept of an $([e],[e]\vee \lbrack
c_{k}])$-ideal is introduced and related properties are investigated. $%
([e],[e]\vee \lbrack c_{k}])$-ideal is a generalized form of $([e],[e]\vee
\lbrack c])$-ideal. Characterizations of $([e],[e]\vee \lbrack c_{k}])$%
-ideals are discussed.
\end{abstract}

\maketitle

\section{\textbf{Introduction}}

A (crisp) set $A$ in a universe $X$ can be defined in the form of its
characteristic function $%
%TCIMACRO{\U{3bc} }%
%BeginExpansion
\mu
%EndExpansion
_{A}:X\rightarrow \{0,1\}$ yielding the value $1$ for elements belonging to
the set $A$ and the value $0$ for elements excluded from the set $A$.

So far most of the generalization of the crisp set have been conducted on
the unit interval $[0,1]$ and they are consistent with the asymmetry
observation. In other words, the generalization of the crisp set to fuzzy
sets spread positive information that fit the crisp point $\{1\}$ into the
interval $[0,1]$.

Because no negative meaning of information is suggested, we now feel a need
to deal with negative information. To do so, we also feel a need to supply a
mathematical tool.

To attain such an object, Jun et al.\cite{Lee} introduced a new function
which is called a negative-valued function, and constructed N-structures.
They applied $N$-structures to BCK/BCI-algebras, and discussed $N$-ideals in
BCK/BCI-algebras. In 1966, Imai and Iseki \cite{Imai} and Iseki \cite{Iseki}
introduced two classes of abstract algebras: BCK-algebras and BCI-algebras.
It is known that the class of BCK-algebras is a proper subclass of the class
of BCI-algebras. As a generalization of a BCK-algebra, Kim and Kim \cite{Kim}
introduced the notion of a BE-algebra, and investigated several properties.
In Ahn and So \cite{Ahn} introduced the notion of ideals in BE-algebras.
They considered several descriptions of ideals in BE-algebras.

M.S. Kang, and Y.B. Jun \cite{Kang}, introduced the notion of an $N$-ideal
of BE-algebra. In paper \cite{Kang}, a point $N$-structure which is
(Conditionally) employed in an $N$-structure is proposed. The concept of $%
([e],[e]\vee \lbrack c])$-ideals and discussed the related properties.

In this paper, the concept of an $([e],[e]\vee \lbrack c_{k}])$-ideal is
introduced and related properties are investigated. $([e],[e]\vee \lbrack
c_{k}])$-ideal is a generalized form of $([e],[e]\vee \lbrack c])$-ideal. In
this paper, a point $N$-structure which is $(k$ Conditionally$)$ employed in
an $N$-structure is introduced. Characterizations of $([e],[e]\vee \lbrack
c_{k}])$-ideals are discussed.

\section{\textbf{Preliminaries}}

\begin{definition}
$($\cite{Kim}$)$ Let $K(\tau )$ be a class of type $\tau =(2,0)$. A system $%
(X;\ast ,1)\in K(\tau )$ define a \textbf{BE-Algebra} if the following
axioms hold:\newline
$(V_{1})$ $(\forall x\in X)$ $(\ x\ast x=1\ )$,\newline
$(V_{2})$ $(\forall x\in X)$ $(\ x\ast 1=1\ )$,\newline
$(V_{3})$ $(\forall x\in X)$ $(\ 1\ast x=x\ )$,\newline
$(V_{4})$ $(\forall x,y,z\in X)$ $(\ x\ast (y\ast z)=y\ast (x\ast z)\ )$.

\begin{definition}
$($\cite{Kang}$)$ A relation "$\leq $" on a BE-algebra $X$ is defined by%
\newline
$(\forall x,y\in X)$ $($ $x\leq y\Leftrightarrow x\ast y=1$ $)$.
\end{definition}
\end{definition}

\begin{definition}
$($\cite{Ahn}$)$ A BE-algebra $X$ is called \textbf{Self-distributive} if $%
x\ast (y\ast z)=(x\ast y)\ast (x\ast z)$ for all $x,y,z\in X$.
\end{definition}

\begin{definition}
$($\cite{Kim}$)$ A BE-algebra $(X;\ast ,1)$ is said to be \textbf{Transitive}
if it satisfies:\newline
$(\forall x,y,z\in X)$ $($ $y\ast z\leq (x\ast y)\ast (x\ast z)$ $)$.
\end{definition}

\textbf{Result: (\cite{Ahn}) }The converse of above proposition is not true
in general.

\textbf{Note: (\cite{Kang}) }The collection of function from a set $X$ to $%
[-1,0]$ is denoted by $\tau (X,[-1,0])$.

\begin{definition}
$($\cite{Ahn}$)$ Let $I$ a non-empty subset of an BE-algebra $X$ then $I$ is
called an \textbf{Ideal} of $X$ if;\newline
$(1)$ $(\forall x\in X$, $s\in I)$ $(\ x\ast s\in I\ )$,\newline
$(2)$ $(\forall x\in X$, $s,q\in I)$ $(\ (s\ast (q\ast x))\ast x\in I\ )$.
\end{definition}

\begin{lemma}
$($\cite{Kang}$)$\label{LEM1} A non-empty subset $I$ of $X$ is an ideal of $%
X $ if and only if it satisfies:\newline
$(1)$ $1\in I$,\newline
$(2)$ $(\forall x,z\in X)$ $($ $\forall y\in I$ $)$ $(\ x\ast (y\ast z)\in
I\Rightarrow x\ast z\in I$ $)\ )$.
\end{lemma}

\section{$N$\textbf{-ideals of BE-algebra}}

\begin{definition}
$($\cite{Kang}$)$ An element of $\tau (X,[-1,0])$ is called a \textbf{%
Negative-valued function} from $X$ to $[-1,0]$ $($briefly, $N$-function on $%
X)$.
\end{definition}

\begin{definition}
$($\cite{Kang}$)$ An ordered pair $(X$,$f)$ of $X$ and an $N$-function $f$
on $X$ is called an $N$\textbf{-structure.}
\end{definition}

\begin{definition}
$($\cite{Kang}$)$ For any $N$-structure $(X,f)$ the nonempty set%
\begin{equation*}
C(f;t):=\{x\in X\text{ }|\text{ }f(x)\leq t\}
\end{equation*}%
is called a \textbf{closed }$(f,t)$\textbf{-cut} of $(X,f)$, where $t\in
\lbrack -1,0]$.
\end{definition}

\begin{definition}
$($\cite{Kang}$)$ By an $N$-ideal of $X$ we mean an $N$-structure $(X,f)$
which satisfies the following condition:\newline
$(\forall t\in \lbrack -1,0])$ $(\ C(f;t)\in J(X)\cup \{\emptyset \}\ )$.%
\newline
Where $J(X)$ is a set of all ideal of $X$.
\end{definition}

\begin{example}
Let $X=\{1,\alpha ,h,m,0\}$ be a set with a multiplication table given by;%
\newline
\begin{equation*}
\begin{tabular}{|l|l|l|l|l|l|}
\hline
$\ast $ & $1$ & $\alpha $ & $h$ & $m$ & $0$ \\ \hline
$1$ & $1$ & $\alpha $ & $h$ & $m$ & $0$ \\ \hline
$\alpha $ & $1$ & $1$ & $\alpha $ & $m$ & $m$ \\ \hline
$h$ & $1$ & $1$ & $1$ & $m$ & $m$ \\ \hline
$m$ & $1$ & $\alpha $ & $h$ & $1$ & $\alpha $ \\ \hline
$0$ & $1$ & $1$ & $\alpha $ & $1$ & $1$ \\ \hline
\end{tabular}%
\end{equation*}%
\newline
Then $(X;\ast ,1)$ is a BE-algebra. Consider an $N$-structure $(X,f)$ in
which $t$ is defined by;\newline
$f(y)=\left\{ 
\begin{array}{ll}
-0.7 & \text{if \ }y\in \{1,\alpha ,h\} \\ 
-0.2 & \text{if \ }y\in \{m,0\}%
\end{array}%
\right. $\newline
Then\newline
$C(f;t)=\left\{ 
\begin{array}{ll}
\{1,\alpha ,h\} & \text{if \ }t\in \lbrack -0.7,0] \\ 
\emptyset & \text{if \ }t\in \lbrack -1,-0.7)%
\end{array}%
\right. $\newline
Note that $\{1,\alpha ,h\}$ is an ideals of BE-algebra $X$, and hence $(X,f)$
is an $N$-ideal of $X$.
\end{example}

\begin{lemma}
Each $N$-ideal $(X,f)$ of BE-algebra $X$ satisfies the condition:\newline
$(\forall x\in X)$ $(\ f(1)\leq f(x)\ )$.
\end{lemma}

\begin{proof}
Since in BE-algebra we have $x\ast x=1$, thus we have $f(1)=f(x\ast x)\leq
f(x)$ for all $x\in X$.
\end{proof}

\begin{proposition}
\label{PRO6}Each $N$-ideal $f$ of BE-algebra $X$ satisfies the condition:%
\newline
$(\forall x,y\in X)$ $(\ f((x\ast y)\ast y)\leq f(x)\ )$.
\end{proposition}

\begin{proof}
Straightforward.
\end{proof}

\begin{proposition}
Each $N$-ideal $f$ of BE-algebra $X$\ satisfies the condition;\newline
$(\forall x,y\in X)$ $(\ f(y)\leq \max \{f(x),f(x\ast y)\}\ )$.
\end{proposition}

\begin{proof}
It can be easily proved.
\end{proof}

\begin{corollary}
If $x\leq y$, then each $N$-ideal $f$ of BE-algebra $X$\ satisfies the
condition;\newline
$f(y)\leq f(x)$.
\end{corollary}

\begin{proof}
Suppose $x\leq y$ for all $x,y\in X$. Then $x\ast y=1$, so 
\begin{equation*}
f(y)=f(1\ast y)=f((x\ast y)\ast y)
\end{equation*}%
By proposition \ref{PRO6}, $f((x\ast y)\ast y)\leq f(x)$, hence $f(y)\leq
f(x)$.\newline
\end{proof}

\section{$([e],[e]\vee \lbrack c_{k}])$-Ideals}

\begin{definition}
$($\cite{Kang}$)$ Let $f$ be an $N$-structure of of BE-algebra $X$ in wich $%
f $ is given by;\newline
$f(y)=\left\{ 
\begin{array}{ll}
0 & \text{if }y\neq x \\ 
t & \text{if \ }y=x%
\end{array}%
\right. $\newline
Where $\ t\in \lbrack -1,0)$, In this case, $f$ is represented by $\frac{x}{t%
}$. $(X,\frac{x}{t})$\ is called \textbf{Point $N$-structure}.
\end{definition}

\begin{definition}
$($\cite{Kang}$)$ A Point $N$-structure $(X,\frac{x}{t})$\ is called \textbf{%
Employed} in an $N$-structure $(X,f)$ of BE-algebra $X$\ if\ \ $f(x)\leq t$
for all $x\in X$, and $t\in \lbrack -1,0)$. It is represented as $(X,\frac{x%
}{t})[e](X,f)$\ or $\frac{x}{t}[e]f$.
\end{definition}

\begin{definition}
A point $N$-structure $(X,\frac{x}{t})$ is called $(k$ \textbf{Conditionally}%
$)$\textbf{\ Employed} in an $N$-structure $(X,f)$ if\ $f(x)+t+k+1<0$ for
all $x\in X$, $t\in \lbrack -1,0)$\ and $k\in (-1,0]$. It is denoted by $(X,%
\frac{x}{t})[c_{k}](X,f)$ or $\frac{x}{t}[c_{k}]f$.\newline
To say that $(X,\frac{x}{t})([e]\vee \lbrack c_{k}])(X,f)$ $($or briefly, $%
\frac{x}{t}([e]\vee \lbrack c_{k}])f)$ we mean $(X,\frac{x}{t})[e](X,f)$ or $%
(X,\frac{x}{t})[c_{k}](X,f)$ $($or briefly, $\frac{x}{t}[e]$ or $\frac{x}{t}%
[c_{k}]f)$. To say that $\frac{x}{t}\overline{\alpha }f$ we mean $\frac{x}{t}%
\alpha f$ does not hold for $\alpha \in \{[e],[c_{k}],[e]\vee \lbrack
c_{k}]\}$.
\end{definition}

\begin{definition}
An $N$-structure $(X,f)$ is called $([e],[e]\vee \lbrack c_{k}])$-ideal of $%
X $ if it satisfied;\newline
$(1)$ $\frac{y}{t}[e]f\Rightarrow \frac{x\ast y}{t}([e]\vee \lbrack c_{k}])f$%
,\newline
$(2)$\ $\frac{x}{t}[e]f$, $\frac{y}{r}[e]f\Rightarrow \frac{(x\ast (y\ast
z))\ast z}{\max \{t,r\}}([e]\vee \lbrack c_{k}])f$.\newline
for all $x,y,z\in X$, where $t,r\in \lbrack -1,0)$\ and $k\in (-1,0]$.
\end{definition}

\begin{example}
Let $X=\{1,\gamma ,0,m,\omega \}$ be a set with a multiplication table given
by;\newline
\begin{equation*}
\begin{tabular}{|l|l|l|l|l|l|}
\hline
$\ast $ & $1$ & $\gamma $ & $0$ & $m$ & $\omega $ \\ \hline
$1$ & $1$ & $\gamma $ & $0$ & $m$ & $\omega $ \\ \hline
$\gamma $ & $1$ & $1$ & $\gamma $ & $m$ & $m$ \\ \hline
$0$ & $1$ & $1$ & $1$ & $m$ & $m$ \\ \hline
$m$ & $1$ & $\gamma $ & $0$ & $1$ & $\gamma $ \\ \hline
$\omega $ & $1$ & $1$ & $\gamma $ & $1$ & $1$ \\ \hline
\end{tabular}%
\end{equation*}%
\newline
Let $(X,f)$ be an $N$-structure. Then $f$ is defined in an $N$-structure $%
(X,f)$, as;\newline
$f=\left( 
\begin{array}{ccccc}
1 & \gamma & 0 & m & \omega \\ 
-0.9 & -0.8 & -0.7 & -0.9 & -0.8%
\end{array}%
\right) $\newline
and$\ t,r\in \lbrack -0.7,-0.3)$, also $k\in (-1,-0.4)$.\newline
for all $x,y,z\in X$, the followings \newline
$(1)$\ $\frac{y}{t}[e]f\Rightarrow \frac{x\ast y}{t}([e]\vee \lbrack
c_{k}])f $,\newline
$(2)$ $\frac{x}{t}[e]f$, $\frac{y}{r}[e]f\Rightarrow \frac{(x\ast (y\ast
z))\ast z}{\max \{t\text{, }r\}}([e]\vee \lbrack c_{k}])f$.\newline
are hold. Hence, $f$\ is an $([e],[e]\vee \lbrack c_{k}])$-ideal of $X$.
\end{example}

\begin{theorem}
\label{TH4}For any $N$-structure $(X,f)$, the following are equivalent:%
\newline
$(1)$ $(X,f)$ is a $([e],[e]\vee \lbrack c_{k}])$-ideal of $X$. \newline
$(2)$ $(X,f)$ satisfies the following inequalities:\newline
$(2.1)$ $(\forall x,y\in X)$ $(\ f(x\ast y)\leq \max \{f(y),\frac{-k-1}{2}%
\}\ )$, \newline
$(2.2)$ $(\forall x,y,z\in X)$ $(\ f((x\ast (y\ast z))\ast z)\leq \max
\{f(x),f(y),\frac{-k-1}{2}\}\ )$. where $k\in (-1,0]$.
\end{theorem}

\begin{proof}
Let $(X,f)$ be a $([e],[e]\vee \lbrack c_{k}])$-ideal of $X$. Suppose that $%
f(x\ast y)>\max \{f(y),\frac{-k-1}{2}\}$ for all $x,y\in X$. If we take $%
t_{y}:=\max \{f(y),\frac{-k-1}{2}\}$, $t_{y}\in \lbrack \frac{-k-1}{2},0]$, $%
\frac{y}{t_{y}}[e]f$ and $\frac{x\ast y}{t_{y}}[\overline{e}]f$. Also,\ $%
f(x\ast y)+t_{y}+k+1>2t_{y}+1\geq 0$, and so $\frac{x\ast y}{t_{y}}[%
\overline{c_{k}}]f$. This is a contradiction. Thus $f(x\ast y)\leq \max
\{f(y),\frac{-k-1}{2}\}$ for all $x,y\in X$. Also suppose that $\ f((x\ast
(y\ast z))\ast z)>\max \{f(x),f(y),\frac{-k-1}{2}\}$ for some $x,y,z\in X$.
Take $t:=\max \{f(x),f(y),\frac{-k-1}{2}\}$. Then $t\geq \frac{-k-1}{2}$,$%
\frac{x}{t}[e]f$ and $\frac{y}{t}[e]f$, but $\frac{x\ast (y\ast z))\ast z}{t}%
[\overline{e}]f$. Also, $f((x\ast (y\ast z)\ast z)+t+k+1>2t+k+1\geq 0$,
i.e., $\frac{x\ast (y\ast z))\ast z}{t}[\overline{c_{k}}]f$. This is a
contradiction, and hence $f((x\ast (y\ast z))\ast z)\leq \max \{f(x),f(y),%
\frac{-k-1}{2}\}$ for all $x,y,z\in X$.\newline
Conversely, suppose that $(X,f)$ satisfies $(2.1)$ and $(2.2)$. Let $\frac{y%
}{t}[e]$ for all $y\in X$ and $t\in \lbrack -1,0)$. Then $f(y)\leq t$.
Suppose that $\frac{x\ast y}{t}[\overline{e}]f$, i.e, $f(x\ast y)>t$. If $%
f(y)>\frac{-k-1}{2}$, then \ $f(x\ast y)\leq \max \{f(y),\frac{-k-1}{2}%
\}=f(y)\leq t$, which is a contradiction. Hence $f(y)\leq \frac{-k-1}{2}$,
which implies that $f(x\ast y)+t+k+1<2f(x\ast y)+k+1\leq 2\max \{f(y),\frac{%
-k-1}{2}\}+k+1=0$, i.e., $\frac{x\ast y}{t}[c_{k}]f$. Thus $\frac{x\ast y}{t}%
([e]\vee \lbrack c_{k}])f$. Let $\frac{x}{t}[e]f$ and $\frac{y}{r}[e]f$ for
all$\ x,y,z\in X$ and $t,r\in \lbrack -1,0)$. Then $f(x)\leq t$ and $%
f(y)\leq r$. Suppose that $\frac{(x\ast y\ast z))\ast z}{\max \{t,r\}}[%
\overline{e}]f$, i.e., $f((x\ast (y\ast z))\ast z)>\max \{t,r\}$. If $\max
\{f(x),f(y)\}>\frac{-k-1}{2}$, then 
\begin{equation*}
f((x\ast (y\ast z))\ast z)\leq \max \{f(x),f(y),\frac{-k-1}{2}\}=\max
\{f(x),f(y)\}\leq \max \{t,r\}\text{.}
\end{equation*}%
This is impossible, and so $\max \{f(x),f(y)\}\leq \frac{-k-1}{2}$. It
follows that $f((x\ast (y\ast z))\ast z)+\max \{t,r\}+k+1<2f((x\ast (y\ast
z))\ast z)+k+1\leq 2\max \{f(x),f(y),\frac{-k-1}{2}\}+k+1=0$\newline
$\Rightarrow \frac{(x\ast y\ast z))\ast z}{\max \{t,r\}}[\overline{c_{k}}]f$%
. Hence $\frac{(x\ast y\ast z))\ast z}{\max \{t,r\}}([e]\vee \lbrack
c_{k}])f $, and therefore $(X,f)$ is a $([e],[e]\vee \lbrack c_{k}])$-ideal
of $X$.
\end{proof}

If $(k=0)$, then the followig holds.

\begin{corollary}
For any $N$-structure $(X,f)$, the following are equivalent:\newline
$(1)$ $(X,f)$ is a $([e],[e]\vee \lbrack c])$-ideal of $X$. \newline
$(2)$ $(X,f)$ satisfies the following inequalities:\newline
$(2.1)$ $(\forall x,y\in X)$ $(\ f(x\ast y)\leq \max \{f(y),-0.5\}\ )$, 
\newline
\end{corollary}

\begin{theorem}
\label{TH5} Every $([e],[e]\vee \lbrack c_{k}])$-ideal $(X,f)$ of an
BE-algebra $X$ satisfies the following inequalities: \newline
$(1)$ $(\forall x\in X)$.$(\ f(1)\leq \max \{f(x),\frac{-k-1}{2}\}\ )$,%
\newline
$(2)$ $(\forall x,y\in X)$ $(\ f((x\ast y)\ast y)\leq \max \{f(x),\frac{-k-1%
}{2}\}\ )$. \ where $k\in (-1,0]$.
\end{theorem}

\begin{proof}
$(1)$: By using $(V_{1})$ and theorem \ref{TH4}$(2.1)$, we have%
\begin{equation*}
f(1)=f(x\ast x)\leq \max \{f(x),\frac{-k-1}{2}\}
\end{equation*}%
for all $x\in X$.\newline
$(2)$: By using $(V_{3})$, we have $f((x\ast y)\ast y)=f((x\ast (1\ast
y))\ast y)$ for all $x,y\in X$\newline
Then by using theorem \ref{TH4}$(2.2)$, we get\newline
$f((x\ast (1\ast y))\ast y)\leq \max \{f(x),f(1),\frac{-k-1}{2}\}=\max
\{f(x),\frac{-k-1}{2}\}$, because by $(1)$ $f(1)\leq \max \{f(x),\frac{-k-1}{%
2}\}$ for all $x,y\in X$.\newline
Hence, $f((x\ast y)\ast y)\leq \max \{f(x),\frac{-k-1}{2}\}$ for all $x,y\in
X$.
\end{proof}

If $(k=0)$, then the followig holds.

\begin{corollary}
Every $([e],[e]\vee \lbrack c])$-ideal $(X,f)$ of an BE-algebra $X$
satisfies the following inequalities: \newline
$(1)$ $(\forall x\in X)$.$(\ f(1)\leq \max \{f(x),-0.5\}\ )$,\newline
$(2)$ $(\forall x,y\in X)$ $(\ f((x\ast y)\ast y)\leq \max \{f(x),-0.5\}\ )$.
\end{corollary}

\begin{corollary}
Each $([e],[e]\vee \lbrack c_{k}])$-ideal $(X,f)$ satisfies the following
condition;\newline
$(\forall x,y\in X)$ $(\ x\leq y\Rightarrow f(y)\leq \max \{f(x),\frac{-k-1}{%
2}\}\ )$. \ \ \ where $k\in (-1,0]$.
\end{corollary}

\begin{proof}
Let $x\leq y$ for all $x,y\in X$. Then $x\ast y=1$,and so%
\begin{equation*}
f(y)=f(1\ast y)=f((x\ast y)\ast y)\leq \max \{f(x),\frac{-k-1}{2}\}
\end{equation*}%
\newline
Hence, $f(y)\leq \max \{f(x),\frac{-k-1}{2}\}$.
\end{proof}

If $(k=0)$, then the followig holds.

\begin{lemma}
Each $([e],[e]\vee \lbrack c])$-ideal $(X,f)$ satisfies the following
condition;\newline
$(\forall x,y\in X)$ $(\ x\leq y\Rightarrow f(y)\leq \max \{f(x),-0.5\}\ )$.
\end{lemma}

\begin{proposition}
\label{PRO2}Let $(X,f)$ be an $N$-structure such that\newline
$(1)$ $(\forall x\in X)$ $(\ f(1)\leq \max \{f(x),\frac{-k-1}{2}\}\ )$,%
\newline
$(2)$ $(\forall x,y,z\in X)$ $(\ f(x\ast z)\leq \max \{f(x\ast (y\ast
z)),f(y),\frac{-k-1}{2}\}\ )$.\newline
Then the following implication is valid.\newline
$(\forall x,y\in X)$ $(\ x\leq y\Rightarrow f(y)\leq \max \{f(x),\frac{-k-1}{%
2}\}\ )$. \ \ where $k\in (-1,0]$.
\end{proposition}

\begin{proof}
Suppose $x\leq y$ for all $x,y\in X$. Then $x\ast y=1$, and by using $(1)$
we get%
\begin{eqnarray*}
f(y) &=&f(1\ast y)\leq \max \{f(1\ast (x\ast y)),f(x),\frac{-k-1}{2}\} \\
&=&\max \{f(1\ast 1),f(x),\frac{-k-1}{2}\} \\
&=&\max \{f(1),f(x),\frac{-k-1}{2}\} \\
&=&\max \{f(x),\frac{-k-1}{2}\}
\end{eqnarray*}%
Hence, $f(y)\leq \max \{f(x),\frac{-k-1}{2}\}$.
\end{proof}

If $(k=0)$, then the followig holds.

\begin{lemma}
Let $(X,f)$ be an $N$-structure such that\newline
$(1)$ $(\forall x\in X)$ $(\ f(1)\leq \max \{f(x),-0.5\}\ )$,\newline
$(2)$ $(\forall x,y,z\in X)$ $(\ f(x\ast z)\leq \max \{f(x\ast (y\ast
z)),f(y),-0.5\}\ )$.\newline
Then the following implication is valid.\newline
$(\forall x,y\in X)$ $(\ x\leq y\Rightarrow f(y)\leq \max \{f(x),-0.5\}\ )$.
\end{lemma}

\begin{theorem}
\label{TH6}Let $(X,f)$ be an $N$-structure of transitive BE-algebra $X$.
Then $(X,f)$ is an $([e],[e]\vee \lbrack c_{k}])$-ideal of $X$ if and only
if it satisfies the following inequalities:\newline
$(1)$ $(\forall x\in X)$ $(\ f(1)\leq \max \{f(x),\frac{-k-1}{2}\}\ )$,%
\newline
$(2)$ $(\forall x,y,z\in X)$ $(\ f(x\ast z)\leq \max \{f(x\ast (y\ast
z)),f(y),\frac{-k-1}{2}\}\ )$. \ \ where $k\in (-1,0]$.
\end{theorem}

\begin{proof}
Suppose that $(X,f)$ is an $([e],[e]\vee \lbrack c])$-ideal of $X$. From
theorem \ref{TH5}$\left( 1\right) $,it is easily seen that%
\begin{equation*}
f(1)\leq \max \{f(x),\frac{-k-1}{2}\}\text{.}
\end{equation*}%
Since $X$ is transitive,%
\begin{equation*}
((y\ast z)\ast z)\ast ((x\ast (y\ast z))\ast (x\ast z))=1\text{ \ \ \ }%
\mathbf{(G)}
\end{equation*}%
for all $x,y,z\in X$. By using $(V_{3})$ and $\mathbf{(G)}$%
\begin{equation*}
f(x\ast z)=f(1\ast (x\ast z))=f(((y\ast z)\ast z)\ast ((x\ast (y\ast z))\ast
(x\ast z))\ast (x\ast z))
\end{equation*}%
\newline
By using theorem \ref{TH4}$(2.2)$, \ref{TH5}$\left( 2\right) $, we have%
\begin{eqnarray*}
f(((y\ast z)\ast z)\ast ((x\ast (y\ast z))\ast (x\ast z))\ast (x\ast z))
&\leq &\max \{f((y\ast z)\ast z),f(x\ast (y\ast z)),\frac{-k-1}{2}\} \\
&=&\max \{f(x\ast (y\ast z)),f((y\ast z)\ast z),\frac{-k-1}{2}\} \\
&\leq &\max \{f(x\ast (y\ast z)),f(y),\frac{-k-1}{2}\}
\end{eqnarray*}%
Hence $f(x\ast z)\leq \max \{f(x\ast (y\ast z)),f(y),\frac{-k-1}{2}\}$ for
all $x,y,z\in X$.\newline
Conversly suppose that $(X,f)$ satisfies $(1)$ and $(2)$. By using $(2)$, $%
(V_{1})$, $(V_{2})$ and $(1)$%
\begin{eqnarray*}
f(x\ast y) &\leq &\max \{f(x\ast (y\ast y)),f(y),\frac{-k-1}{2}\} \\
&=&\max \{f(x\ast 1),f(y),\frac{-k-1}{2}\} \\
&=&\max \{f(1),f(y),\frac{-k-1}{2}\} \\
&=&\max \{f(y),\frac{-k-1}{2}\}
\end{eqnarray*}%
Also by using $(2)$ and $(1)$ we get%
\begin{eqnarray*}
f((x\ast y)\ast y) &\leq &\max \{f((x\ast y)\ast (x\ast y)),f(x),\frac{-k-1}{%
2}\} \\
&=&\max \{f(1),f(x),\frac{-k-1}{2}\} \\
&=&\max \{f(x),\frac{-k-1}{2}\}
\end{eqnarray*}%
for all $x,y\in X$. Now, since $(y\ast z)\ast z\leq (x\ast (y\ast z))\ast
(x\ast z)$ for all $x,y,z\in X$, it follows that from proposition \ref{PRO2}%
, we have\newline
\begin{equation*}
f((x\ast (y\ast z))\ast (x\ast z))\leq \max \{f((y\ast z)\ast z),\frac{-k-1}{%
2}\}
\end{equation*}%
\newline
So, from $(2)$, we have%
\begin{eqnarray*}
f((x\ast (y\ast z))\ast z) &\leq &\max \{f((x\ast (y\ast z))\ast (x\ast
z)),f(x),\frac{-k-1}{2}\} \\
&\leq &\max \{f((y\ast z)\ast z),f(x),\frac{-k-1}{2}\} \\
&\leq &\max \{f(x),f(y),\frac{-k-1}{2}\}
\end{eqnarray*}%
for all $x,y,z\in X$. Using theorem \ref{TH4}, we conclude that $(X,f)$ is a 
$([e],[e]\vee \lbrack c])$-ideal of $X$.
\end{proof}

If $(k=0)$, then the followig holds.

\begin{corollary}
Let $(X,f)$ be an $N$-structure of transitive BE-algebra $X$. Then $(X,f)$
is an $([e],[e]\vee \lbrack c])$-ideal of $X$ if and only if it satisfies
the following inequalities:\newline
$(1)$ $(\forall x\in X)$ $(\ f(1)\leq \max \{f(x),-0.5\}\ )$,\newline
$(2)$ $(\forall x,y,z\in X)$ $(\ f(x\ast z)\leq \max \{f(x\ast (y\ast
z)),f(y),-0.5\}\ )$.
\end{corollary}

\begin{theorem}
Let $X$ be a transitive BE-algebra. If $(X,f)$ is a $([e],[e]\vee \lbrack
c_{k}])$-ideal of $X$ such that $f(1)>\frac{-k-1}{2}$, then $(X,f)$ is an $N$%
-ideal of $X$. \ \ where $k\in (-1,0]$.
\end{theorem}

\begin{proof}
Suppose that $(X,f)$ is a $([e],[e]\vee \lbrack c_{k}])$-ideal of $X$ such
that $\frac{-k-1}{2}<f(1)$. Then $\frac{-k-1}{2}<f(x)$ and so $\frac{-k-1}{2}%
<f(1)\leq f(x)$ for all $x\in X$ by theorem \ref{TH6}$(1)$%
\begin{equation*}
f(1)\leq \max \{f(x),\frac{-k-1}{2}\}
\end{equation*}%
for all $x\in X$. It follows that from theorem \ref{TH6}$(2)$,%
\begin{eqnarray*}
f(x\ast z) &\leq &\max \{f(x\ast (y\ast z)),f(y),\frac{-k-1}{2}\} \\
&=&\max \{f(x\ast (y\ast z)),f(y)\}
\end{eqnarray*}%
for all $x,y,z\in X$. Hence $(X,f)$ is an $N$-ideal of $X$.
\end{proof}

If $(k=0)$, then the followig holds.

\begin{corollary}
Let $X$ be a transitive BE-algebra. If $(X,f)$ is a $([e],[e]\vee \lbrack c])
$-ideal of $X$ such that $f(1)>-0.5$, then $(X,f)$ is an $N$-ideal of $X$.
\end{corollary}

\begin{theorem}
If $(X,f)$ is a $([e],[e]\vee \lbrack c_{k}])$-ideal of a transitive
BE-algebra $X$. Show that%
\begin{equation*}
(\forall t\in \lbrack -1,\frac{-k-1}{2}))\text{ }(Q(f;t)\in J(X)\cup
\{\emptyset \})
\end{equation*}%
\ \ where $Q(f;t):=\{x\in X$ $|$ $\frac{x}{t}[c_{k}]f\}$, $J(X)$ is a set of
all ideal of $X$ and $k\in (-0.5,0]$.
\end{theorem}

\begin{proof}
\begin{corollary}
Suppose that $Q(f;t)\neq \emptyset $ for all $t\in \lbrack -1,\frac{-k-1}{2})
$. Then there exists $x\in Q(f;t)$, and so $\frac{x}{t}[c]f$, i.e., $%
f(x)+t+k+1<0$. Using theorem \ref{TH6}$(1)$, we have%
\begin{eqnarray*}
f(1) &\leq &\max \{f(x),\frac{-k-1}{2}\} \\
&=&\left\{ 
\begin{array}{ll}
\frac{-k-1}{2} & \text{if \ }f(x)\leq \frac{-k-1}{2} \\ 
f(x) & \text{if \ }f(x)>\frac{-k-1}{2}%
\end{array}%
\right.  \\
&<&-1-t-k
\end{eqnarray*}%
which indicates that $1\in Q(f;t)$. Let $x\ast (y\ast z)\in Q(f;t)$ for all $%
x,y,z\in X$ here $y\in Q(f;t)$. Then $\frac{x\ast (y\ast z)}{t}[c_{k}]f$ and 
$\frac{y}{t}[c]f$, i.e., $f(x\ast (y\ast z))+t+k+1<0$ and $f(y)+t+k+1<0$.
Using theorem \ref{TH6}$(2)$, we get%
\begin{equation*}
f(x\ast z)\leq \max \{f(x\ast (y\ast z)),f(y),\frac{-k-1}{2}\}
\end{equation*}%
Thus, if $\max \{f(x\ast (y\ast z)),f(y)\}>\frac{-k-1}{2}$, then%
\begin{equation*}
f(x\ast z)\leq \max \{f(x\ast (y\ast z)),f(y)\}<-1-t-k
\end{equation*}%
If\ $\max \{f(x\ast (y\ast z)),f(y)\}\leq \frac{-k-1}{2}$, then $f(x\ast
z)\leq \frac{-k-1}{2}<-1-t-k$. This show that $\frac{x\ast z}{t}[c_{k}]f$
i.e., $x\ast z\in Q(f;t)$. By using lemma \ref{LEM1}, we have $Q(f;t)$ is an
ideal of $X$.
\end{corollary}
\end{proof}

If $(k=0)$, then the followig holds.

\begin{corollary}
If $(X,f)$ is a $([e],[e]\vee \lbrack c])$-ideal of a transitive BE-algebra $%
X$. Show that%
\begin{equation*}
(\forall t\in \lbrack -1,-0.5))\text{ }(Q(f;t)\in J(X)\cup \{\emptyset \})
\end{equation*}%
\ \ where $Q(f;t):=\{x\in X$ $|$ $\frac{x}{t}[c]f\}$, and $J(X)$ is a set of
all ideal of $X$
\end{corollary}

\begin{theorem}
Let $X$ be a transitive BE-algebra. Then the followings are equivalent:%
\newline
$(1)$ An $N$-structure $(X,f)$ is a $([e],[e]\vee \lbrack c_{k}])$-ideal of $%
X$\newline
$(2)$ $(\forall t\in \lbrack -1,0))$ $([f]_{t}\in J(X)\cup \{\emptyset \})$%
\newline
where $[f]_{t}:=C(f;t)\cup \{x\in X$ $|$ $f(x)+t+k+1\leq 0\}$, $J(X)$ is a
set of all ideal of $X$, and $k\in (-1,0]$.
\end{theorem}

\begin{proof}
$(1)\Rightarrow (2)$: Suppose that $(1)$ satisfies. Let $[f]_{t}\neq
\emptyset $, here $t\in \lbrack -1,0)$. Then there exists $x\in \lbrack
f]_{t}$, and so $f(x)\leq t$ \ or $f(x)+t+k+1\leq 0$ for all $x\in X$ and $%
t\in \lbrack -1,0)$. If $f(x)\leq t$, then\newline
\begin{eqnarray*}
f(1) &\leq &\max \{f(x),\frac{-k-1}{2}\}\leq \max \{t,\frac{-k-1}{2}\} \\
&=&\left\{ 
\begin{array}{ll}
t & \text{if \ }t>\frac{-k-1}{2} \\ 
\frac{-k-1}{2}\leq -1-t-k & \text{if \ }t\leq \frac{-k-1}{2}%
\end{array}%
\right. 
\end{eqnarray*}%
\newline
By theorem \ref{TH6}$(1)$. Hence $1\in \lbrack f]_{t}$. If $f(x)+t+k+1\leq 0$%
, then\newline
\begin{eqnarray*}
f(1) &\leq &\max \{f(x),\frac{-k-1}{2}\}\leq \max \{-1-t-k,\frac{-k-1}{2}\}
\\
&=&\left\{ 
\begin{array}{ll}
-1-t-k & \text{if \ }t<\frac{-k-1}{2} \\ 
\frac{-k-1}{2}\leq t & \text{if \ }t\geq \frac{-k-1}{2}%
\end{array}%
\right. 
\end{eqnarray*}%
\newline
And so $1\in \lbrack f]_{t}$. Let $x,y,z\in X$ be such that $y\in \lbrack
f]_{t}$ and $x\ast (y\ast z)\in \lbrack f]_{t}$. Then $f(y)\leq t$ or $%
f(y)+t+k+1\leq 0$, and \ $f(x\ast (y\ast z))\leq t$ or $f(x\ast (y\ast
z))+t+k+1\leq 0$. Thus we let the four cases:\newline
$(a_{1})$ $f(y)\leq t$ and $f(x\ast (y\ast z))\leq t$,\newline
$(a_{2})$ $f(y)\leq t$ and $f(x\ast (y\ast z))+t+k+1\leq 0$,\newline
$(a_{3})$ $f(y)+t+k+1\leq 0$ and $f(x\ast (y\ast z))\leq t$,\newline
$(a_{4})$ $f(y)+t+k+1\leq 0$ and $f(x\ast (y\ast z))+t+k+1\leq 0$.\newline
For case $(a_{1})$, theorem \ref{TH6}$(2)$, implies that%
\begin{eqnarray*}
f(x\ast z) &\leq &\max \{f(x\ast (y\ast z)),f(y),\frac{-k-1}{2}\}\leq \max
\{t,\frac{-k-1}{2}\} \\
&=&\left\{ 
\begin{array}{ll}
\frac{-k-1}{2} & \text{if \ }t<\frac{-k-1}{2} \\ 
\text{ }t & \text{if \ }t\geq \frac{-k-1}{2}%
\end{array}%
\right. 
\end{eqnarray*}%
so that $x\ast z\in C(f;t)$ or $f(x\ast z)+t+k\leq \frac{-k-1}{2}+\frac{-k-1%
}{2}+k=-1$. Thus $x\ast z\in \lbrack f]_{t}$. For case $(a_{2})$, we have%
\begin{eqnarray*}
f(x\ast z) &\leq &\max \{f(x\ast (y\ast z)),f(y),\frac{-k-1}{2}\}\leq \max
\{-1-t-k,t,\frac{-k-1}{2}\} \\
&=&\left\{ 
\begin{array}{ll}
-1-t-k & \text{if \ }t<\frac{-k-1}{2} \\ 
t & \text{if \ }t\geq \frac{-k-1}{2}%
\end{array}%
\right. 
\end{eqnarray*}%
Thus $x\ast z\in \lbrack f]_{t}$.\newline
For case $(a_{3})$, the prove is same to case $(a_{2})$. For case $(a_{4})$
we have,%
\begin{eqnarray*}
f(x\ast z) &\leq &\max \{f(x\ast (y\ast z)),f(y),\frac{-k-1}{2}\}\leq \max
\{-1-t-k,\frac{-k-1}{2}\} \\
&=&\left\{ 
\begin{array}{ll}
-1-t-k & \text{if \ }t<\frac{-k-1}{2} \\ 
\frac{-k-1}{2} & \text{if \ }t\geq \frac{-k-1}{2}%
\end{array}%
\right. 
\end{eqnarray*}%
So that, $\ x\ast z\in \lbrack f]_{t}$. By using lemma \ref{LEM1}, $[f]_{t}$
is an ideal of $X$.\newline
$(2)\Rightarrow (1)$: Suppose that $(2)$ hold. If $f(1)>\max \{f(y),\frac{%
-k-1}{2}\}$ for all $y\in X$, then $f(1)>t_{y}\geq \max \{f(y),\frac{-k-1}{2}%
\}$ for some $t_{y}\in \lbrack \frac{-k-1}{2},0)$. It follows that $x\in
C(f;t_{y})\subseteq \lbrack f]_{t_{y}}$ but $1\notin C(f;t_{y})$. Also, $%
f(1)+t_{y}+k+1>2t_{y}+k+1\geq 0$. Hence $1\notin \lbrack f]_{t_{y}}$, which
contradicts the supposition. So, $f(1)\leq \max \{f(y),\frac{-k-1}{2}\}$ for
all $y\in X$. Suppose that\ for some $x,z\in X$, we have\newline
$f(x\ast z)>\max \{f(x\ast (y\ast z)),f(y),\frac{-k-1}{2}\}$ \ \ \ \ \ \ \ \
\ \ \ $\mathbf{(D)}$\newline
Taking\ $t:=\max \{f(x\ast (y\ast z)),f(y),\frac{-k-1}{2}\}$\ implies that $%
t\in \lbrack \frac{-k-1}{2},0)$, $x\in C(f;t)\subseteq \lbrack f]_{t}$, and $%
x\ast (x\ast z)\in C(f;t)\subseteq \lbrack f]_{t}$. Since $[f]_{t}$ is an
ideal of $X$, we have $x\ast z\in \lbrack f]_{t}$, and so $f(x\ast z)\leq t$
or $f(x\ast z)+t+k+1\leq 0$. The inequality $\mathbf{(D)}$ induces $x\ast
z\notin C(f;t)$, and $f(x\ast z)+t+k+1>2t+k+1\geq 0$. Thus $x\ast z\notin
\lbrack f]_{t}$. It contradicts the supposition. Hence $f(x\ast z)\leq \max
\{f(x\ast (y\ast z)),f(y),\frac{-k-1}{2}\}$ for all $x,y,z\in X$. Using
theorem \ref{TH6}, we have, $(X,f)$ is a $([e],[e]\vee \lbrack c_{k}])$%
-ideal of $X$.
\end{proof}

If $(k=0)$, then the followig holds.

\begin{corollary}
Let $X$ be a transitive BE-algebra. Then the followings are equivalent:%
\newline
$(1)$ An $N$-structure $(X,f)$ is a $([e],[e]\vee \lbrack c])$-ideal of $X$%
\newline
$(2)$ $(\forall t\in \lbrack -1,0))$ $([f]_{t}\in J(X)\cup \{\emptyset \})$%
\newline
where $[f]_{t}:=C(f;t)\cup \{x\in X$ $|$ $f(x)+t+1\leq 0\}$, and $J(X)$ is a
set of all ideal of $X$
\end{corollary}

\section{\textbf{Conclusion:}}

In this paper, we have investigated the $([e],[e]\vee \lbrack c_{k}])$%
-ideals of BE-algebra by using transitive and distributive BE-algebra, their
related properties, and provide characterizations of $([e],[e]\vee \lbrack
c_{k}])$-ideals in an $N$-structure $(X,f)$.

Now by using these results we can deal with negative informations, also by
using these results we will be able to solve the difficulties of theories
such as probability theory, ideal theory, algebras theory. In this paper, we
give the new mathematical tools for dealing with uncertainties.

\end{document}